\documentclass[9pt,a4paper]{article}
\usepackage{amsmath}
\usepackage{amsfonts}

\usepackage[]{amsthm,mathtools,braket,enumerate}
\usepackage[top=22truemm,bottom=28truemm,left=20truemm,right=20truemm]{geometry}
\usepackage{ascmac}
\usepackage{color}
\usepackage{amssymb}
\usepackage{bm}
\usepackage{setspace}
\usepackage{multicol}
\usepackage[a]{esvect}
\usepackage{fancybox}
\usepackage{hyperref,graphicx}
\usepackage{multirow}
\usepackage[all]{xy} 
\usepackage{enumerate}
\hypersetup{colorlinks=false,bookmarks=true,bookmarksnumbered=true,pdfborder={0 0 0},bookmarkstype=toc}
\usepackage{titlesec}
\usepackage{boites}
\usepackage{tikz}
\usetikzlibrary{intersections, calc, arrows.meta, graphs}
\tikzstyle{v} = [circle, draw, inner sep=2pt, minimum size=3pt, fill=black]

\newcommand{\maru}[1]{\raise0.2ex\hbox{\textcircled{\scriptsize{#1}}}}
\newcommand{\minimaru}[1]{\raise0.2ex\hbox{\textcircled{\tiny{#1}}}}

\newcommand{\bigset}[2]{\Bigl\{ \, {#1} \, \, \Bigr| \, \, {#2} \, \, \Bigr\}}

\newcommand{\G}{\mathcal{G}}
\newcommand{\bG}{{\bar{\mathcal{G}}}}
\newcommand{\Tref}[1]{(T-\ref{#1})}
\newcommand{\Cref}[1]{(C-\ref{#1})}

\newcommand{\priod}{.\,}
\newcommand{\ceq}{\coloneqq}

\newcommand{\BL}[1]{\textcolor{black}{{#1}}}
\newcommand{\T}{\mathcal{T}}
\newcommand{\C}{\mathcal{C}}

\newcommand{\resp}{resp\priod}

\titleformat*{\section}{\large\bfseries}
\titleformat*{\subsection}{\normalsize\bfseries}

\theoremstyle{plain} 
\newtheorem{theorem}{Theorem}[section]
\newtheorem{proposition}[theorem]{Proposition}
\newtheorem{lemma}[theorem]{Lemma}
\newtheorem{corollary}[theorem]{Corollary}

\theoremstyle{definition}
\newtheorem{definition}[theorem]{Definition}
\newtheorem{example}[theorem]{Example}

\title{Signed Graphs and Signed Cycles of Hyperoctahedral Groups}
\author{Ryo Uchiumi
\thanks{Department of Mathematics, Graduate School of Science, Hokkaido University, Sapporo, Hokkaido 060-0810, Japan. }}
\date{}
\begin{document}

\maketitle

\begin{abstract}

For a graph with edge ordering, a linear order on the edge set, we obtain a permutation of vertices by considering the edges as transpositions of endvertices. 
It is known from D\'enes' results that the permutation of a tree is a full cyclic for any edge ordering. 
As a corollary, D\'enes counted up the number of representations of a full cyclic permutation by means of product of the minimal number of transpositions.
Moreover, a graph with an edge ordering which the permutation is a full cyclic is characterized by graph embedding. 

In this article, we consider an analogy of these results for signed graphs and hyperoctahedral groups. We give a necessary and sufficient condition for a signed graph to have an edge ordering such that the permutation is an even (or odd) full cyclic. We show that the edge ordering of the signed tree with some loops always gives an even (or odd) full cyclic permutation and count up the number of representations of an odd full cyclic permutation by means of product of the minimal number of transpositions.
\end{abstract}


\fontsize{10pt}{15pt}\selectfont
\section{Introduction}
\subsection{background}

Let $n$ be a positive integer. Suppose that ${\mathcal{G}}=(V_\G,E_\G)$ is a graph and $V_\G=[n] \ceq \{1,\ldots,n\}$. We associate an edge $e=\{i,j\} \in E_\G$ with the transposition $\tau_e \ceq (i\ j) \in \mathfrak{S}_n$, where $\mathfrak{S}_n$ denotes the symmetric group of degree $n$.

An \textit{edge ordering} of a graph $\G$ $\leq_\omega$ is a linear order on $E_\G$, denoted as a sequence $\omega=(e_1,\ldots,e_m)$ in which $e_i \leq_\omega e_j$ if $i \leq j$. For an edge ordering $\omega=(e_1,\ldots,e_m)$, we give the product $\pi_\omega \ceq \tau_{e_m}\cdots\tau_{e_1} \in \mathfrak{S}_n$.

\begin{definition}
An edge ordering $\omega$ is a \textit{full cyclic permutation ordering} if the product $\pi_\omega \in \mathfrak{S}_n$ is a cyclic permutation of length $n$.
\end{definition}

We can characterize a graph with a full cyclic permutation ordering by studying the orbit and the sign of $\pi_\omega$.

\begin{proposition}
If a graph $\G$ has a full cyclic permutation ordering, then $\G$ is connected and the Betti number $\beta(\G) \ceq |E_\G|-|V_\G|+1$ is even.

\end{proposition}

For graphs with full cyclic permutation ordering, the following theorem by D\'enes is known.
\begin{theorem}[D\'enes {\cite[\S{2}, Theorem 1]{denes1959representation-potmiothaos}}. See also {\cite[\S{2}, Lemma]{moszkowski1989solution-ejoc}},  {\cite[\S{2}, Lemma 2.1]{pawlowski2022chromatic-ac}}]\label{t11}\label{t12}\label{denes}
Given a graph $\G$, the following are equivalent.
\begin{enumerate}[(i)]
\item Any edge ordering of $\G$ is a  full cyclic permutation ordering.
\item $\mathcal{G}$ is a tree.
\end{enumerate}
\end{theorem}

D\'enes gives this theorem to count up the number of representations of a cyclic permutation of length $n$ by means of a product of a minimal number of transpositions, and obtains the following corollary.

\begin{corollary}[D\'enes {\cite[\S{2}, Corollary]{denes1959representation-potmiothaos}}]\label{c13}\label{denescor}
The number of representations of a cyclic permutation of length $n$ by means of a product of $n-1$ transpositions is $n^{n-2}$.
\end{corollary}

The author and Tsujie \cite{tsujie2022upper} also characterize graphs having a full cyclic permutation ordering in terms of graph embedding.


\subsection{The hyperoctahedral group}

In this article, we discuss these arguments in the case of the hyperoctahedral group $\mathfrak{H}_n$, the Weyl group of type $B_n$.  It is a natural idea since $\mathfrak{S}_n$ is known as the Weyl group of type $A_{n-1}$.

Let $I_n \ceq \{-n,\ldots,-1,1,\ldots,n\}$.  The \textit{hyperoctahedral group} $\mathfrak{H}_n$ is a subgroup of the symmetric group $\mathfrak{S}_{I_n}$ of $I_n$, defined by 
\[
\mathfrak{H}_n \ceq \bigset{\eta \in \mathfrak{S}_{I_n}}{\text{$\eta(-i)=-\eta(i)$ for all $i \in I_n$}}.
\]
An element $\eta \in \mathfrak{H}_n$, called a \textit{signed permutation}, is expressed as
\[
\eta=\begin{pmatrix}
1&2&\cdots&n\\
\eta(1)&\eta(2)&\cdots&\eta(n)
\end{pmatrix}
\]
since $\eta$ is determined by values at integers in $[n]$.
In particular, given a permutation in $\mathfrak{S}_n$, the signed permutation in $\mathfrak{H}_n$ is naturally determined. Thus we can consider $\mathfrak{S}_n$ to be a subgroup of $\mathfrak{H}_n$.

The hyperoctahedral groups are well studied (cf. {\cite{borovik2010mirrors,carter1972conjugacy-cm,humphreys1990reflection,kerber1971representations,kerber1975representations,young1930quantitative-potlms-1}}). Here we recall some that are necessary.

There are three types of \textit{signed transpositions} in $\mathfrak{H}_n$. For $i,j$ be distinct numbers in $[n]$ and $\epsilon \in \{+,-\}$, which we consider as a multiplicative group of order two in the natural way, we say that
\[
(i\ \epsilon j) \ceq \begin{pmatrix}
1&\cdots&i&\cdots&j&\cdots&n\\
1&\cdots&\epsilon j&\cdots&\epsilon i&\cdots&n
\end{pmatrix}
\]
is a \textit{positive transposition} if $\epsilon=+$ and a \textit{negative transposition} if $\epsilon=-$.  Note that we can consider a positive transposition  $(i\ {+j})$ as a transposition in $\mathfrak{S}_n$ and abbreviate it as $(i\ j)$. For $i \in [n]$, we say that
\[
(i\ {-i}) \ceq \begin{pmatrix}
1&\cdots&i&\cdots&n\\
1&\cdots&{-i}&\cdots&n
\end{pmatrix}
\]
is an \textit{inversion transposition}. 

The finite group $\mathfrak{H}_n$ is generalized by these signed transpositions, but only some positive transpositions and one inversion transposition are sufficient for the generator. Also, The group $\mathfrak{H}_n$ is isomorphic to the semidirect product $\mathfrak{S}_2^n \rtimes \mathfrak{S}_n$ (the wreath product $\mathfrak{S}_2 \wr \mathfrak{S}_n$). The group $\mathfrak{S}_2^n$ is regarded as the group generalized by inversion transpositions. 
Let $\eta \in \mathfrak{H}_n$, using the relation for signed transpositions given below, we can find $\sigma_1 \in  \mathfrak{S}_2^n$ and $\sigma_2 \in \mathfrak{S}_n$ such that $\eta=\sigma_1\sigma_2$ (more on this later Example $\ref{example}$).

\begin{lemma}\label{Tlemma}
Let $i,j,k \in [n]$ be distinct integers and $\epsilon \in \{+,-\}$.  Then the following three claims hold:
\begin{enumerate}[(T-1)]
\item $(i\ {-j})=(i\ {-i})(j\ {-j})(i\ {j})$; \label{T1}
\item $(i\ \epsilon{j})(j\ {-j})=(i\ {-i})(i\ \epsilon{j})$; \label{T2}
\item $(i\ \epsilon{j})(k\ {-k})=(k\ {-k})(i\ \epsilon{j})$. \label{T3}
\end{enumerate}
\end{lemma}

A signed permutation is decomposed uniquely into a product of commuting signed cycles. Let $i_1,\ldots,i_l$ be distinct numbers in $[n]$ and $\epsilon_1,\ldots,\epsilon_l \in \{+,-\}$, given such a \textit{signed cycle} of length $l$
\begin{align}
\sigma=\begin{pmatrix}
i_1&i_2&\cdots&i_{l-1}&i_l\\
\epsilon_1i_2&\epsilon_2i_3&\cdots&\epsilon_{l-1}i_l&\epsilon_li_1\label{s1}
\end{pmatrix},
\end{align}
we say that $\sigma$ is an \textit{even $l$-cycle} if $\epsilon_1\cdots\epsilon_l=+$ and an \textit{odd $l$-cycle} if $\epsilon_1\cdots\epsilon_l=-$. In the case of $l=n$, the signed cycle $\sigma$ is called an \textit{even} (resp\priod \textit{odd}) \textit{full cyclic permutation}. Note that a cyclic permutation in $\mathfrak{S}_n$ is an even cyclic permutation in $\mathfrak{H}_n$. In addition, formula $\eqref{s1}$ is written as
\begin{align}
\sigma=\bigl(i_1\ \varepsilon_1i_2\ \varepsilon_2i_3\ \cdots\cdots\ \varepsilon_{l-1}i_l\bigr)_{\varepsilon_l},\label{s2}
\end{align}
where $\varepsilon_i \ceq \epsilon_1\cdots\epsilon_i$. Also, we obtain the representation of an even cycle as the product of transpositions as follows:
\begin{align}
\bigl(i_1\ \varepsilon_1i_2\ \varepsilon_2i_3\ \cdots\cdots\ \varepsilon_{l-1}i_l\bigr)_+ = (i_1\ \epsilon_1i_2)(i_2\ \epsilon_2i_3) \cdots\cdots (i_{l-1}\ \epsilon_{l-1}i_l).\label{s3}
\end{align}
We know the following relation for signed cycles and inversion transpositions.

\begin{lemma}\label{l23}\label{Clemma}
Suppose that $i_1,\ldots,i_l \in I_n$ have different absolute values and $\varepsilon \in \{+,-\}$.  Then the following three claims hold:
\begin{enumerate}[(C-1)]
\item $\bigl(i_1\ i_2\ \cdots\ i_l\bigr)_\varepsilon=\bigl(i_2\ \cdots\ i_l\ \varepsilon i_1\bigr)_\varepsilon$; \label{C1}
\item $(i_1\,{-i_1})\bigl(i_1\ \cdots\ i_l\bigr)_\varepsilon=\bigl(i_1\ \cdots\ i_l\bigr)_\varepsilon(i_l\ {-}i_l) = \bigl(i_1\ \cdots\ i_l\bigr)_{-\varepsilon}$. \label{C2}

\end{enumerate}
In other words, multiplying a signed cycle by one appropriate inversion transposition changes the even-oddness.

\end{lemma}

The representation of a signed permutation as the product of signed cycles determines the conjugacy class in $\mathfrak{H}_n$. The conjugacy class of $\mathfrak{H}_n$ are parameterized by pair of two integer partition $(\lambda,\mu)$ such that $|\lambda|+|\mu|=n$, 
and those containing even and odd $n$-cycle are represented by $((n),0)$ and $(0,(n))$, respectively.

\subsection{Signed graph}

We introduce the signed graph to consider an analogy of D\'enes' results for $\mathfrak{H}_n$. 
Here, define a \textit{signed graph} as a quadruple $\mathcal{G}=(V_\G,E^+_\G,E^-_\G,L_\G)$, where $V_\G=[n]$,\ $E^+_\G$ and $E^-_\G$ are collections consisting of unordered pairs of elements in $V_\G$, and $L_\G$ is a subset of $V_\G$. 
An element in $V_\G$ is called a \textit{vertex}, an element in $E^+_\G$ (resp\priod $E^-_\G$) is called a \textit{positive} (resp\priod \textit{negative}) \textit{edge}, and an element in $L_\G$ is called a \textit{loop}. Denote the set $E_\G^+ \sqcup E_\G^- \sqcup L_\G$ by $E_\G$, whose element is called an \textit{edge}.

For a signed graph $\G=(V_\G,E^+_\G,E^-_\G,L_\G)$, let $\bG=(V_\bG,E_\bG)$ denote an unsigned multigraph such that $V_\bG=V_\G,\ E_\bG=E_\G^+ \sqcup E_\G^-$.
A \textit{signed tree} is a graph $\T=(V_\T,E_\T^+,E_\T^-,L_\T)$ such that $E_\T^+ \cap E_\T^-=L_\T=\emptyset$ and the unsigned graph $\bar{\T}$ is a tree. 

For every edge $e \in E_\G$, define a signed transposition $\tau_e \in \mathfrak{H}_n$ by
\[
\tau_e \ceq \begin{cases}
(i\ {j})&\text{if } e=\{i,j\} \in E^+_\G,\\
(i\ {-j})&\text{if } e=\{i,j\} \in E^-_\G,\\
(i\ {-i})&\text{if } e=i \in L_\G.
\end{cases}
\]

Let $\omega=(e_1,\ldots,e_m)$ be an edge ordering of a signed graph $\G$ (a linear order on $E_\G$). We give a product $\pi_\omega \ceq \tau_{e_m}\cdots\tau_{e_1} \in \mathfrak{H}_n$.

\begin{definition}
An edge ordering of $\G$ $\omega$ is an \textit{even} (resp\priod \textit{odd}) \textit{full cyclic permutation ordering} if $\pi_\omega$ is an even (resp\priod odd) full cyclic permutation.
\end{definition}

\subsection{Main results}
The main results are as follows:

\begin{theorem}\label{nsc}
\BL{Given a signed graph $\G$, the following are equivalent.
\begin{enumerate}[(i)]
\item $\G$ has an even (\resp odd) full cyclic permutation ordering.
\item $\bG$ has a full cyclic permutation ordering and $|L_\G|$ is even (\resp odd).
\end{enumerate}}
\end{theorem}

\begin{theorem}\label{t2}\label{t15}
\BL{Given a signed graph $\G$, the following are equivalent.
\begin{enumerate}[(i)]
\item Any edge ordering of $\G$ is an even (resp\priod odd)  full cyclic permutation ordering.
\item $\bG$ is a tree and $|L_\G|$ is even (resp\priod odd), that is, $\G$ is a signed tree with even (resp odd) loops.
\end{enumerate}}
\end{theorem}

\begin{corollary}\label{c4}\label{c16}
\textit{
\begin{enumerate}[(A)]
\item The number of representations of an even $n$-cycle by means of a product of $n-1$ transpositions is $n^{n-2}$.
\item The number of representations of an odd $n$-cycle by means of a product of $n$ transpositions is $n^n$.
\end{enumerate}}
\end{corollary}

More general results are known about Corollary $\ref{c4}$ . Let $W$ be a well-generalized complex reflection group of rank $n$ with Coxeter number $h$. Then the number of factorizations of a fixed Coxeter element $c \in W$ into a product of $n$ reflections is given by the following formula (see {\cite[\S1, Formula (1.1)]{chapuy_counting_2014-1}}):
\[
\#\bigset{(\tau_n,\ldots,\tau_1)}{\text{$\tau_i$ is a reflection in $W$,\ the product $\tau_n\cdots\tau_1$ is equal to $c$}} = \dfrac{n!}{|W|}h^n.
\]

For the hyperoctahedral group $\mathfrak{H}_n$, since $|\mathfrak{H}_n|=2^nn!$ and $h=2n$, the value of the above formula is $n^n$. In this paper, we prove Corollary $\ref{c4}$ by counting the number of sequences consisting of $n-1$ (resp\priod $n$) transpositions such that the product is an even (resp\priod odd) $n$-cycle.

Note that Chapuy and Stump {\cite[\S1, Theorem 1.1]{chapuy_counting_2014-1}} give the formula for the exponential generating function of factorizations of a fixed Coxeter element into a product of reflections.

The organization of this paper is as follows.
In Section $\ref{sec2}$, we discuss the representation of $\pi_\omega$ and prove Theorem $\ref{nsc}$.
In Section $\ref{sec3}$, we prove Theorem $\ref{t2}$ and Corollary $\ref{c4}$.

\section{Edge orderings} \label{sec2}
\subsection{The representation of the product obtained from the edge ordering}

In this subsection, we prove a key lemma about the representation of the permutation $\pi_\omega$. 

Let $\G=(V_\G,E^+_\G,E^-_\G,L_\G)$ be a signed graph.
\begin{definition}
For an edge ordering $\omega$ of $\G$, define $\varphi(\omega)$ as an edge ordering of $\bG$ obtained by excluding loops from $\omega$. For an edge ordering $\bar{\omega}$ of $\bG$, define $\varphi^{-1}(\bar{\omega})$ as the set of edge orderings $\omega$ of $\G$ such that $\varphi(\omega)=\bar{\omega}$.
\end{definition}

Define the projection
\[
\mathit{\Phi}:\mathfrak{H}_n \cong \mathfrak{S}_2^n \rtimes \mathfrak{S}_n \ni (\sigma_1,\sigma_2) \mapsto \sigma_2 \in \mathfrak{S}_n.
\]
Let $\omega$ be an edge ordering of $\G$. Since
\[
\mathit{\Phi}\bigl((i\ {+}j)\bigr) = \mathit{\Phi}\bigl((i\ {-}j)\bigr) = (i\ j),\quad \mathit{\Phi}\bigl((i\ {-}i)\bigr) = 1,
\]
where $i,j$ are distinct integers in $[n]$ and $1$ is the identity element of $\mathfrak{S}_n$, the permutation $\pi_{\varphi(\omega)}$ is equal to $\mathit{\Phi}(\pi_\omega)$.

\begin{lemma}\label{l24}\label{reprelemma}
For a signed graph $\G$, let $\omega$ be an edge ordering of $\G$. There exist some inversion transpositions $\nu_1,\ldots,\nu_r \in \mathfrak{H}_n$ such that $\pi_\omega=\nu_1\cdots\nu_r\pi_{\varphi(\omega)}$ and $r \equiv |L_\G| \pmod 2$.
\end{lemma}
\begin{proof}
Let $\omega=(e_1,\ldots,e_m)$ denote an edge ordering of $\G$. Since $\pi_\omega \in \mathfrak{H}_n \cong \mathfrak{S}_2^n \rtimes \mathfrak{S}_n$, there exist $\sigma_1 \in \mathfrak{S}_2^n$ and $\sigma_2 \in \mathfrak{S}_n$ uniquely such that $\pi_\omega=\sigma_1\sigma_2$.  We see that $\sigma_2=\mathit{\Phi}(\pi_\omega)=\pi_{\varphi(\omega)}$.
To show the latter of the statement, define the group homomorphism $\mathit{\Psi}:\mathfrak{H}_n \to \{1,-1\}$ by
\[
\mathit{\Psi}(\tau) = \begin{cases}
-1 &\text{if $\tau$ is an inversion transposition},\\
1&\text{otherwise}
\end{cases}
\]
for a signed transposition $\tau \in \mathfrak{H}_n$, where $\{1,-1\}$ is a multiplicative group of order two in the natural way. Then we have
\[
\mathit{\Psi}(\pi_\omega)=\mathit{\Psi}(\tau_{e_1})\cdots\mathit{\Psi}(\tau_{e_m})=(-1)^{|L_\G|},
\]
where $\omega=(e_1,\ldots,e_m)$. On the other hand, since $\sigma_1 \in \mathfrak{S}_2^n$, there exist inversion transpositions $\nu_1,\ldots,\nu_r$ such that $\sigma_1=\nu_1\cdots\nu_r$. Hence we get  
\[
\mathit{\Psi}(\pi_\omega)=\mathit{\Psi}(\sigma_1)\mathit{\Psi}(\sigma_2) = \mathit{\Psi}(\nu_1)\cdots\mathit{\Psi}(\nu_r) = (-1)^r.
\]
Thus $r \equiv |L_\G| \pmod 2$.
\end{proof}

\begin{example}\label{example}
\textrm{Assume that $\pi_\omega=(1\ 2)(2\ {-}2)(2\ {-}3)(3\ {-}4)(5\ {-}5)(3\ 5)$, then $\pi_{\varphi(\omega)}=(1\ 2)(2\ 3)(3\ 4)(3\ 5)$. The representation given by Lemma $\ref{reprelemma}$ is obtained in the following steps using Lemma $\ref{Tlemma}$:}
\begin{enumerate}[Step 1.]
\item Move all inversion transpositions contained in $\pi_\omega$ to the left by \Tref{T2} or \Tref{T3}.\label{st1}
\item Change all negative transpositions contained in $\pi_\omega$ to the product of two inversion transpositions and one positive transposition by \Tref{T1}.\label{st2}
\item Move all inversion transpositions resulting from Step $\ref{st2}$ to the left by \Tref{T2} or \Tref{T3}.\label{st3}
\end{enumerate}

Therefore, we obtain
\begin{align}
\pi_\omega &=(1\ 2)(2\ {-}2)(2\ {-}3)(3\ {-}4)(5\ {-}5)(3\ 5)\nonumber\\
&=(1\ {-}1)(5\ {-}5) \cdot (1\ 2)(2\ {-}3)(3\ {-}4)(3\ 5)\tag{by Step $\ref{st1}$}\\
&=(1\ {-}1)(5\ {-}5)\cdot(1\ 2)\cdot(2\ {-}2)(3\ {-}3)(2\ 3)\cdot(3\ {-}3)(4\ {-}4)(3\ 4)\cdot(3\ 5)\tag{by Step $\ref{st2}$}\\
&=(1\ {-}1)(5\ {-}5)(1\ {-}1)(3\ {-}3)(1\ {-}1)(4\ {-}4)\cdot (1\ 2)(2\ 3)(3\ 4)(3\ 5)\tag{by Step $\ref{st3}$}\\
&=(1\ {-}1)(3\ {-}3)(4\ {-}4)(5\ {-}5)\cdot \pi_{\varphi(\omega)}.\nonumber
\end{align}
\end{example}

\subsection{The condition to have a full cyclic permutation ordering}

Now, we prove one of the main theorems, the condition for a signed graph to have a full cyclic permutation ordering.

\begin{theorem}[Restatement of Theorem $\ref{nsc}$]\label{tnsc}
\BL{Given a signed graph $\G$, the following are equivalent.
\begin{enumerate}[(i)]
\item $\G$ has an even (\resp odd) full cyclic permutation ordering.
\item $\bG$ has a full cyclic permutation ordering and $|L_\G|$ is even (\resp odd).
\end{enumerate}}
\end{theorem}

We devise this theorem in two propositions. 

\begin{proposition}\label{p27}
Let $\G$ be a signed graph. If $\omega$ is an even (\resp odd) full cyclic permutation ordering of $\G$, then $\varphi(\omega)$ is a full cyclic permutation ordering of $\bG$ and $|L_\G|$ is even (\resp odd).
\end{proposition}
\begin{proof}
Suppose that $\omega$ is an even (\resp odd) full cyclic permutation ordering. By Lemma $\ref{reprelemma}$, there exist inversion transpositions $\nu_1,\ldots,\nu_r$ such that $\pi_\omega=\nu_1\cdots\nu_r\pi_{\varphi(\omega)}$ and $r \equiv |L_\G| \pmod 2$. We know that $\varphi(\omega)$ is a full cyclic permutation ordering as follows by $\eqref{s2}$, $\eqref{s3}$ and Lemma $\ref{Clemma}$:
\begin{align*}
\pi_{\varphi(\omega)} = \mathit{\Phi}(\pi_\omega) &= \mathit{\Phi}(\bigl(i_1\ \varepsilon_1i_2\ \varepsilon_2i_3\ \cdots\cdots\ \varepsilon_{n-1}i_n\bigr)_\varepsilon) \\
&= \mathit{\Phi}((i_1\ \varepsilon i_1)(i_1\ \epsilon_1i_2)(i_2\ \epsilon_2i_3) \cdots\cdots (i_{n-1}\ \epsilon_{n-1}i_n))\\
&= (i_1\ i_2)(i_2\ i_3) \cdots\cdots (i_{n-1}\ i_n)\\
&= \bigl(i_1\ i_2\ \cdots\cdots\ i_n\bigr)_+,
\end{align*}
where $i_1,\ldots,i_n$ are distinct integers in $[n]$, $\epsilon_1,\ldots,\epsilon_{n-1} \in \{+,-\}$, $\varepsilon_i=\epsilon_1\cdots\epsilon_i$. 

Moreover, we see that $r$ is even (\resp odd) by Lemma $\ref{Clemma}$. Thus $|L_\G|$ is also even (\resp odd).
\end{proof}

\begin{proposition}
Let $\G$ be a signed graph. If $\bar{\omega}$ is a full cyclic permutation ordering of $\bG$ and $|L_\G|$ is even (\resp odd), then any edge ordering in $\varphi^{-1}(\bar{\omega})$ is an even (\resp odd) full cyclic permutation ordering of $\G$.
\end{proposition}
\begin{proof}

Assume that $|L_\G|$ is even (\resp odd) and let $\bar{\omega}$ be a full cyclic permutation ordering of $\bG$. For any edge ordering $\omega'$ in $\varphi^{-1}(\bar{\omega})$, there exist inversion transpositions $\nu_1,\ldots,\nu_r$ such that  
\[
\pi_{\omega'}=\nu_1\cdots\nu_r\pi_{\varphi(\omega')}=\nu_1\cdots\nu_r\pi_{\bar{\omega}}
\]
by Lemma $\ref{reprelemma}$, where $r$ is even (\resp odd). 
Since $\pi_{\bar{\omega}}$ is an even full cyclic permutation in $\mathfrak{H}_n$, we know that $\pi_{\omega'}$ is an even (\resp odd) full cyclic permutation by Lemma $\ref{Clemma}$. Hence $\omega'$ is an even (\resp odd) full cyclic permutation ordering of $\G$.
\end{proof}

Since an unsigned graph with a full cyclic permutation ordering is connected, we get the following corollary:
\begin{corollary}\label{ca}
If a signed graph $\G$ has an odd full cyclic permutation ordering and $|E_\G|=n$, then $\bG$ is a tree and $|L_\G|=1$, that is, $\G$ is a signed tree with one loop.
\end{corollary}

\section{Signed trees and full cyclic permutations}\label{sec3}

Now, we prove the remaining main theorems. 

\subsection{The proof of Theorem $\ref{t15}$}

\begin{theorem}[Restatement of Theorem $\ref{t15}$]
{Given a signed graph $\G$, the following are equivalent.
\begin{enumerate}[(i)]
\item Any edge ordering of $\G$ is an even (resp\priod odd)  full cyclic permutation ordering.\label{t311}
\item $\bG$ is a tree and $|L_\G|$ is even (resp\priod odd), that is, $\G$ is a signed tree with even (resp odd) loops.\label{t312}
\end{enumerate}}
\end{theorem}


First we show $\eqref{t311} \Rightarrow \eqref{t312}$. We know that if a signed graph $\G$ has an even (\resp odd) full cyclic permutation ordering, then $|L_\G|$ is even (\resp odd) by Proposition $\ref{p27}$. 
Therefore we need to show the following.

\begin{proposition}
If any edge ordering of $\G$ is an even (\resp odd)  full cyclic permutation ordering, then $\bG$ is a tree.
\end{proposition}
\begin{proof}
We prove the contraposition. Suppose that $\bG$ is not a tree. Since a signed tree with a full cyclic permutation ordering is connected, we can assume that $\bG$ has a cycle. Let $\C=(V_\C,E_\C)$ denote the minimal cycle (minimal number of vertices) that $\bG$ has, where
\[
V_\C=\{v_1,\ldots,v_l\},\quad E_\C = \{c_1,\ldots,c_l\},
\]
with $c_i \ceq \{v_i,v_{i+1}\}$ ($v_{l+1}=v_1$). Let $I_{\bG\backslash\C}(v_l)$ denote the set of edges of $\bG$ that are incident to the vertex $v_l$ but not contained in $E_\C$. Then we see that there exists no edge in $I_{\bG\backslash\C}(v_l)$ such that it is incident to $v_1$. If $b_0 \in I_{\bG\backslash\C}(v_l)$ is incident to $v_1$, then $\bG$ has a cycle of length $2$ consisting of $b_0$ and $c_l$. We have $l=2$ by minimality of $\mathcal{C}$. Thus there exist three edges between $v_1$ and $v_l$. It  contradicts to the construction of $\bG$.

Define an edge ordering of $\bG$ $\omega_0$ by
\[
\omega_0 \ceq (c_1,\ldots,c_{l-1},d_1,\ldots,d_t,c_l,b_1,\ldots,b_s),
\]
where
\[
I_{\bG\backslash\C}(v_l) = \{b_1,\ldots,b_s\},\quad  E_\G\backslash (I_{\bG\backslash\C}(v_l) \cup E_\C) = \{d_1,\ldots,d_t\}.
\]
Then $E_\G = \{b_1,\ldots,b_s,d_1,\ldots,d_t,c_1,\ldots,c_l\}$. We obtain
\begin{align*}
\pi_{\omega_0}(v_1) 
&= \tau_{b_s}\cdots\tau_{b_1}\tau_{c_l}\tau_{d_t}\cdots\tau_{d_1}\tau_{c_{l-1}}\cdots\tau_{c_1}(v_1)\\
&= \tau_{b_s}\cdots\tau_{b_1}\tau_{c_l}\tau_{d_t}\cdots\tau_{d_1}(v_l)\\
&= \tau_{b_s}\cdots\tau_{b_1}\tau_{c_l}(v_l)\\
&= \tau_{b_s}\cdots\tau_{b_1}(v_1)\\
&= v_1.
\end{align*}
Hence $\omega_0$ is not a full cyclic permutation ordering. Therefore edge orderings in $\varphi^{-1}(\omega_0)$ are neither even nor odd full cyclic permutation orderings of $\G$.
\end{proof}

Now, we show $\eqref{t312} \Rightarrow \eqref{t311}$ by induction on the number of vertices $n$. We can suppose that $n \geq 2$. Suppose that $\bG$ is a tree and $|L_\G|$ is even (\resp odd). Let $\omega=(e_1,\ldots,e_m)$ be an edge ordering of $\G$.
Assume that $e_1$ is a loop. Since $\G$ is a connected graph with at least two vertices, there exists an edge of $\G$ which is not a loop. Let  $\omega'=(e_{k},\ldots,e_m,e_1,\ldots,e_{k-1})$ denote an edge ordering of $\G$ such that $e_k$ is not a loop. Since
\[
\pi_{\omega'}= \tau_{e_{k-1}}\cdots\tau_{e_1}\tau_{e_m}\cdots\tau_{e_{k}} =
(\tau_{e_{k-1}}\cdots\tau_{e_1})\pi_{\omega}(\tau_{e_{k-1}}\cdots\tau_{e_1})^{-1},
\]
that is, $\pi_\omega$ and $\pi_{\omega'}$ are conjugate, $\omega$ is a full cyclic permutation ordering if and only if $\omega'$ is a full cyclic permutation ordering.
Hence we can assume $e_1$ is not a loop. 

Let $e_1=\{i_1,j_1\}$ and $\G\backslash{e_1}$ denote the graph excluding the edge $e_1$ from $\G$. The graph $\G\backslash{e_1}$ has two connected components $\T_1$ and $\T_2$ which are signed trees with some loops. We assume that $\T_1$ includes $i_1$ and $\T_2$ includes $j_1$. Let $\omega_1=(c_{1},\ldots,c_{p})$  and $\omega_2=(d_{1},\ldots,d_{q})$ denote the edge orderings of $\T_1$ and $\T_2$ which are obtained by restricting $\omega$ to $E_{\T_1}$ and $E_{\T_2}$, where $p+q=m-1$. 
Since $\omega_1$ and $\omega_2$ are full cyclic permutation orderings by induction hypothesis, we can write $\pi_{\omega_1}$ and $\pi_{\omega_2}$ as
\begin{align*}
\pi_{\omega_1} &= \nu_1\cdots\nu_{s}\bigl(i_1\ i_2\ \cdots\cdots\ i_l\bigr)_+,\\
\pi_{\omega_2} &= \mu_{1}\cdots\mu_{t}\bigl(j_1\ j_2\ \cdots\cdots\ j_k\bigr)_+,
\end{align*}
where $\nu_1,\ldots,\nu_s,\mu_1,\ldots,\mu_t$ are distinct inversion transpositions, $i_1,\ldots,i_l,j_1,\ldots,j_k$ are distinct integers in $[n]$, $s+t \equiv |L_\G| \pmod 2$ and $l+k=n$ by Lemma $\ref{l24}$.

Since the signed transposition corresponding to an edge of $\T_1$ and the signed transposition corresponding to an edge of $\T_2$ are commutative, we obtain 
\begin{align*}
\pi_\omega &= \pi_{\omega_1}\pi_{\omega_2}\tau_{e_1}\\
&= \nu_1\cdots\nu_{s}\mu_{1}\cdots\mu_{t}\bigl(i_1\ i_2\ \cdots\cdots\ i_l\bigr)_+\bigl(j_1\ j_2\ \cdots\cdots\ j_k\bigr)_+(i_1\ \epsilon j_1)\\
&= \nu_1'\cdots\nu_r' \bigl(i_1\ i_2\ \cdots\cdots\ i_l\bigr)_+\bigl(j_1\ j_2\ \cdots\cdots\ j_k\bigr)_+(i_1\  j_1)\\
&= \nu_1'\cdots\nu_r' \bigl(i_1\ j_2\ \cdots\ j_k\ j_1\ i_2\ \cdots\ i_l\bigr)_+,
\end{align*}
where $\nu_1',\ldots,\nu_r'$ are inversion transpositions and $r \equiv |L_\G| \pmod 2$. Thus $\omega$ is an even (\resp odd) full cyclic permutation ordering if $|L_\G|$ is even (\resp odd).

Now, Theorem $\ref{t15}$ is completely proven. \qed

\subsection{The proof of Corollary $\ref{c16}$}

\begin{corollary}[Restatement of Corollary $\ref{c16}$]
\ 
\begin{enumerate}[(A)]
\item The number of representations of an even $n$-cycle by means of a product of $n-1$ transpositions is $n^{n-2}$.\label{A}
\item The number of representations of an odd $n$-cycle by means of a product of $n$ transpositions is $n^n$.\label{B}
\end{enumerate}
\end{corollary}

We can obtain $\eqref{A}$ as the corollary of Corollary $\ref{c13}$.
Now, we prove $\eqref{B}$. 

Let $x$ be the number of representations of an odd full cyclic permutation by means of a product of $n$ transpositions. Suppose that 
\[
X \ceq \bigset{(\tau_n,\ldots,\tau_1)}{\text{$\tau_i$ is a signed transposition,\ the product $\tau_n\cdots\tau_1$ is an odd full cyclic permutation}}.
\]
Since the number of odd full cyclic permutations is $(n-1)!\cdot 2^{n-1}$, we know $|X|=x\cdot (n-1)!\cdot 2^{n-1}$. 

Define the set consisting of signed graphs by
\begin{align*}
Y \ceq \bigset{\G=(V_\G,E_\G^+,E_G^-,L_\G)}{V_\G=[n],\ |E_\G|=n,\ \text{$\G$ has an odd full cyclic permutation ordering}}
\end{align*}
and define the set consisting of pairs of graphs and edge orderings of them by
\begin{align*}
Z \ceq \bigset{(\G,\omega)}{\G \in Y,\ \text{$\omega$ is an odd full cyclic permutation ordering of $\G$}}.
\end{align*}
Since we have
\[
Y=\bigset{\G=(V_\G,E_\G^+,E_G^-,L_\G)}{V_\G=[n],\ |L_\G|=1,\ \text{$\bG$ is a tree}}
\]
by Corollary $\ref{ca}$, we obtain
\[
Z = \bigset{(\G,\omega)}{\G \in Y,\ \text{$\omega$ is an edge ordering of $\G$}}.
\]

Since $|Y|=n^{n-2}\cdot n\cdot 2^{n-1}$ (Cayley's formula), we get $|Z|=|Y|\cdot n! = n^{n-1}\cdot 2^{n-1}\cdot n!$. Also, since there exists a bijection between $X$ and $Z$, we see $|X|=|Z|$. Hence $x=n^n$.  \qed

\section*{Acknowledgment}
The author wishes to thank everyone who advised for this paper by reading the archives and joining in the discussion.

\bibliographystyle{amsplain}
\bibliography{BCnbibtex}

\end{document}